\documentclass[pdflatex]{sn-jnl}
\usepackage{float}
\usepackage{verbatim}
\usepackage{graphicx}
\usepackage{multirow}
\usepackage{amsmath,amssymb,amsfonts}
\usepackage{amsthm}
\usepackage{mathrsfs}
\def\h{\hfill\triangle}
\usepackage[title]{appendix}
\usepackage{xcolor}
\usepackage{textcomp}
\usepackage{manyfoot}
\usepackage{booktabs}
\usepackage{algorithm}
\usepackage{algorithmicx}
\usepackage{tikz}
\usepackage{caption}
\usepackage{algpseudocode}
\usepackage{listings}
\usepackage{pgfplots}
\pgfplotsset{compat=1.18}
\newtheorem{Theorem}{Theorem}[section]

\newtheorem{Lemma}[Theorem]{Lemma}

\newtheorem{Definition}[Theorem]{Definition}
\newtheorem{Example}[Theorem]{Example}
\begin{document}

\title[Convergence Rates for the Alternating Minimization Algorithm in Structured Nonsmooth and Nonconvex Optimization]
{Convergence Rates for the Alternating Minimization Algorithm in Structured Nonsmooth and Nonconvex Optimization}

\author*[1]{\fnm{Glaydston C.} \sur{Bento}}\email{glaydston@ufg.br}

\author[2]{\fnm{Boris S.} \sur{Mordukhovich}}\email{aa1086@wayne.edu}
\equalcont{These authors contributed equally to this work.}

\author[1]{\fnm{Tiago S.} \sur{Mota}}\email{tiago.mota@discente.ufg.br}
\equalcont{These authors contributed equally to this work.}

\author[3]{\fnm{Antoine} \sur{Soubeyran}}\email{antoine.soubeyran@gmail.com}
\equalcont{These authors contributed equally to this work.}

\affil*[1]{\orgdiv{Institute of Mathematics and Statistics, IME}, \orgname{Federal University of Goi\'as}, \orgaddress{\city{Goi\^ania}, \postcode{74690-900}, \state{Goi\'as}, \country{Brazil}}}

\affil[2]{\orgdiv{Department of Mathematics and Institute for Artificial Intelligence and Data Science}, \orgname{Wayne State University}, \orgaddress{\city{Detroit}, \state{Michigan}, \country{United States}}}

\affil[3]{\orgdiv{Aix Marseille Univ},  \orgaddress{\street{CNRS}, \city{AMSE}, \postcode{610101}, \state{Marseille}, \country{France}}}

\abstract{This paper is devoted to developing the alternating minimization algorithm for problems of structured nonconvex optimization proposed by Attouch, Bolt\'e,  Redont, and Soubeyran in 2010. Our main result provides significant improvements of the convergence rate of the algorithm, especially under the 
low exponent Polyak-Łojasiewicz-Kurdyka condition when we establish either finite termination of this algorithm or its superlinear convergence rate instead of the previously known linear convergence. We also investigate the PLK exponent calculus and  discuss applications to noncooperative games and behavioral science.}

\maketitle

\section{Introduction}\label{sec1}

In this paper, we consider a class of minimization problems for cost functions $L:\mathbb{R}^n \times\ \mathbb{R}^m \to \overline{\mathbb{R}}:=\mathbb{R}\cup \{\infty\}$ given in the structured form of constrained optimization:
\begin{equation}\label{L}
\mbox{minimize }\;L(x, y): = f(x) + Q(x, y) + g(y),
\end{equation}
where $f$ and $g$ are proper lower semicontinuous functions on $\mathbb{R}^n$, not necessarily convex, and where $Q$ is a smooth function that couples the variables $x$ and $y$. Recall that a function $h\colon\mathbb{R}^n\to\overline{\mathbb{R}}$ is {\em proper} if ${\rm dom}\,h:=\{x\in\mathbb{R}^n\;|\;h(x)<\infty\}\ne\emptyset$.

It has been well recognized that optimization problems in the structural form \eqref{L} appear in many important applications including, in particular, noncooperative games, machine learning, image reconstruction, approximation theory, behavior science, etc.; see \cite{attouch2008alternating,attouch2010proximal,attouch2007new} and the references therein.

Our main attention in this paper is paid to the {\em alternative minimization algorithm} designed and investigated in \cite{attouch2010proximal} for nonconvex structured problems of type \eqref{L}. The algorithm's dynamics are described as follows. Let \( (x_0, y_0) \in \mathbb{R}^n \times \mathbb{R}^m \), the sequence is generated iteratively of the form $(x_k, y_k) \to (x_{k+1}, y_k) \to (x_{k+1}, y_{k+1})$ by
\begin{equation}\label{proximalx}
x_{k+1} \in \arg\min \left\{ L(u, y_k) + \frac{1}{2 \lambda_k} \| u - x_k \|^2\;\Big|\;u \in \mathbb{R}^n \right\},
\end{equation}
\begin{equation}\label{proximaly}
y_{k+1} \in \arg\min \left\{ L(x_{k+1}, v) + \frac{1}{2 \mu_k} \| v - y_k \|^2\;\Big|\;v\in \mathbb{R}^m \right\}.
\end{equation}
where $\{\lambda_k\}$ and $\{\mu_k\}$ are positive sequences. This algorithm was originally developed in \cite{attouch2008alternating,attouch2007new} for convex and weakly convex problems, while its nonconvex generality explored in \cite{attouch2010proximal} opened new perspectives and significantly extended the spectrum of important applications to various optimization-related areas.  

The main results of \cite{attouch2010proximal} revolve around establishing global convergence and convergence rates of the iterates in \eqref{proximalx} and \eqref{proximaly} under some versions of the "Kurdyka-\L ojasiewicz inequality" as labeled in \cite{attouch2010proximal}, which we now unify---by following \cite{bento2025convergence}---under the name of the {\em Polyak-\L ojasiewicz-Kurdyka $(PLK)$ conditions}; see Section~\ref{sec:plk} for the exact formulation, historical remarks, and discussions. 

Here we provide significant improvements of convergence rates from \cite[Theorem~11]{attouch2010proximal} establishing, in particular, either finite termination or superlinear convergence for the alternating minimization algorithm under the lower exponent PLK condition instead of the linear convergence obtained in \cite{attouch2010proximal}. Moreover, new convergence rate results are derived for cost function values along the iterative sequence. We also investigate calculus rules for PLK exponents and construct several examples illustrating the revealed phenomena. Finally, some applications to noncooperative games and behavioral science are briefly discussed.\vspace*{0.05in}

The rest of the paper is organized as follows. In Section~\ref{sec:plk}, we first present the basic assumptions and then formulate and discuss the PLK conditions used in deriving the main results. Section~\ref{sec2} investigates calculus rule for PLK components needed to deal with the structured form \eqref{L}. We construct here five examples illustrating some striking PLK phenomena. Section~\ref{sec3} presents our main results about convergence rates of the alternating minimization algorithm under the exponent PLK conditions. Section~\ref{sec:appl} concerns applications of the main results to some models of noncooperative game theory and behavioral science.

\section{Basic Assumptions and PLK Conditions}\label{sec:plk}

First we formulate the standing assumptions that are used throughout the paper without further mentioning. They are exactly the same as in \cite{attouch2010proximal}, namely:
\[
(\mathcal{H})
\begin{cases}
L(x,y) = f(x) + Q(x, y) + g(y), & \\
f: \mathbb{R}^n \to\overline{\mathbb{R}}, & \text{is proper lower semicontinuous,} \\
g: \mathbb{R}^m \to \mathbb{R}\to\overline{\mathbb{R}}, & \text{is proper lower semicontinuous,} \\
Q: \mathbb{R}^n \times \mathbb{R}^m \to \mathbb{R}, & \text{is a } {\cal C}^1-\text{smooth function,} \\
\nabla Q &\hspace{-1cm} \text{is Lipschitz continuous on bounded subsets of } \mathbb{R}^n \times \mathbb{R}^m,
\end{cases}
\]
\[
(\mathcal{H}_1)
\begin{cases}
\inf_{\mathbb{R}^n \times \mathbb{R}^m} L > -\infty, \\
\text{the function } L(\cdot, y_0) \text{ is proper}, \\
\text{for some positive } r_- < r_+, \text{ the sequences of stepsizes } \lambda_k, \mu_k \text{ belong to } (r_-, r_+).
\end{cases}
\]\vspace*{0.05in}

It is shown in \cite[Lemma~3.1]{attouch2010proximal} that the imposed assumptions ensure that the alternating minimization algorithm in \eqref{proximalx}, \eqref{proximaly} is well-defined and enjoys the desired descent property together with other well-posedness properties broadly used in the paper. To present this result, we need to recall the main subdifferential notion for extended-real-valued proper functions used in what follows. Given $h\colon\mathbb{R}^n\to\overline{\mathbb{R}}$ and $\bar x\in{\rm dom}\,h$, the {\em limiting/Mordukhovich subdifferential} of $h$ at $\bar x$ is defined by 
\begin{equation}\label{lsub}
\partial_M h(\bar{x}):=\big\{v\in\mathbb{R}^n\;\big|\;\exists x_k\stackrel{h}{\to}\bar{x}, \,v_k\in\partial^Fh(x_k),\;v_k\to v\big\},
\end{equation}
where $x_k\stackrel{h}{\to}\bar{x}$ means that $x_k\to \bar{x}$ and $h(x_k)\to f(x)$, and where 
\begin{equation}\label{rsub}
\partial_F h(x):=\Big\{v\in\mathbb{R}^n\;\Big|\;\limsup_{u\to x}\frac{h(u)-h(x)-\langle v,u-x\rangle}{\|u-x\|}\ge 0\big\}
\end{equation}
is known as the {\em regular/Fr\'echet subdifferential} of $h$ at $x$. It has been realized in variational analysis and optimization that the limiting subdifferential \eqref{lsub} is a robust construction enjoying comprehensive {\em calculus rules} based on variational and extremal principles. On the other hand, the regular subdifferential \eqref{rsub} doesn't possess these properties while providing convenient approximation tools for computations. We refer the reader to the monographs \cite{mordukhovich2006variational,mordukhovich2018variational,rockafellar1998}
and the bibliographies therein for variational theories involving \eqref{lsub}, \eqref{rsub} in finite and infinite dimensions with numerous applications.\vspace*{0.05in}

The aforementioned result of \cite{attouch2010proximal} is formulated as follows.

\begin{Lemma}\label{basic-lem}
Under assumptions $(\mathcal{H})$ and $(\mathcal{H}_1)$, the sequences $\{x_k\}$ and $\{y_k\}$ are well-defined. Moreover, the following hold:  

{\bf(i)}  For all natural numbers $k\in\mathbb{N}$, we have the inequality
\begin{equation}\label{decrescimo}
L(x_k, y_k) + \frac{1}{2 \lambda_{k-1}} \|x_k - x_{k-1}\|^2 + \frac{1}{2 \mu_{k-1}} \|y_k - y_{k-1}\|^2 \leq L(x_{k-1}, y_{k-1}),
\end{equation}
and therefore the sequence $\{L(x_k,y_k)\}$ is not increasing. 

{\bf(ii)}  We have the convergent series
\[
\sum_{k=1}^{\infty} \left( \|x_k - x_{k-1}\|^2 + \|y_k - y_{k-1}\|^2 \right) < \infty,
\]
which implies, in particular, that 
\[
\lim_{k \to \infty} (\|x_k - x_{k-1}\| + \|y_k - y_{k-1}\|) = 0.
\]

{\bf(iii)}  For all $k\in\mathbb{N}$, define  
\[
(x^*_k, y^*_k): = \left( \nabla_x Q(x_k, y_k) - \nabla_x Q(x_k, y_{k-1}), 0 \right) - \left( \frac{1}{\lambda_{k-1}} (x_k - x_{k-1}), \frac{1}{\mu_{k-1}} (y_k - y_{k-1}) \right).  
\]
Then we have the inclusion
\[
(x^*_k, y^*_k) \in\partial L(x_k, y_k).  
\]
Furthermore, it follows for any bounded subsequence $\{(x_{k_j}, y_{k_j})\}$ of $\{(x_k,y_k)\}$ that
$
(x^*_{k_j}, y^*_{k_j}) \to(0,0) \quad \text{as} \quad k_j \to \infty
$  
and hence $\text{\rm dist}(0, \partial L(x_{k_j}, y_{k_j})) \to 0 \quad \text{as} \quad k_j \to \infty.
$
\end{Lemma}

Now we are ready to formulate the main conditions for our convergence analysis; cf.\ \cite{attouch2010proximal,bento2025convergence} for more details and references.

\begin{Definition}\label{def:kl} Let $h:\mathbb{R}^m \to\overline{\mathbb R}$ be an extended-real-valued lower semicontinuous function. We say that the function $h$ satisfies the $(basic)$ {\sc Polyak-\L ojasiewicz (PLK) condition} at $\bar{x}\in{\rm dom}\,h$ if there exist a number $\eta \in (0,\infty)$, a neighborhood $U$ of $\bar{x}$, and a concave continuous function $\varphi:[0,\eta] \to[0,\infty)$, called the {\sc desingularizing function}, such that
\begin{equation}\label{plk}
\varphi(0)=0,\quad\varphi\in {\cal C}^1(0,\eta),\quad\varphi'(s)>0\;\mbox{ for all }\; s\in(0,\eta),\;\mbox{ and }\;
\end{equation}
\begin{equation}\label{desigualdadeklND}
\varphi'\big(h(x)-h(\bar{x})\big)\mbox{\rm dist}\big(0,\partial h(x)\big)\ge 1\;\mbox{ for all }\;x\in U\cap[h (\bar x) <h(x)<h(\bar{x})+\eta].
\end{equation}\\[-0.5ex]
The {\sc exponent PLK condition} corresponds to \eqref{plk} and \eqref{desigualdadeklND} where the desingularizing function is chosen as $\varphi(t)=M t^{1-q}$ with some $M>0$ and $q\in[0,1)$. We refer to the case where $q\in(0,1/2)$ as the {\sc PLK condition with lower exponents}.
\end{Definition}

To the best of our knowledge, the original version of the PLK condition was introduced by Polyak in 1962 (the English translation of his paper was published in \cite{polyak1963gradient}), for ${\cal C}^1$-smooth functions $h$ with Lipschitzian gradients in Hilbert spaces, as the inequality
\begin{equation*}
\|\nabla h(x)\|\ge(1/2M)|h(x)-h(\Bar{x})|^{1/2},\quad M>0.
\end{equation*}
The main motivation and result (Theorem~4 of \cite{polyak1963gradient}) were to establish the linear convergence of the classical gradient descent method. Independently, \L ojasiewicz \cite{Lojasiewicz1963} introduced the inequality`
\begin{equation}\label{pl}
\|\nabla h(x)\|\geq b\,|h(x)-h(\Bar{x})|^q,\quad b:=1/M(1-q),\quad q\in[0,1),
\end{equation}
for analytic functions in the finite-dimensional framework of semialgebraic geometry with no applications to optimization. The gradient inequality \eqref{pl} is referred to (especially in the literature on machine learning and computer science) as the Polyak-\L ojasiewicz condition; see, e.g., \cite{karimi2016linear}. Later \cite{Kurdyka}, Kurdyka  extended the semialgebraic approach by \L ojasiewicz to the general class of o-minimal structures. Conditions of the PLK type were further extended to nonsmooth functions in terms of various subdifferentials (often under the name of the ``Kurdyka-\L ojasiewicz property") with many applications to constrained optimization and other important fields of mathematics and applied sciences. Among numerous publications using PLK conditions, we refer the reader to \cite{Absil2005,aragon2023coderivative,aragon2020boosted,attouch2009convergence,attouch2013convergence,bento2025convergence,bento2015generalized,mordukhovich2024} and the bibliographies therein. 

\section{Calculus of PLK Exponents}\label{sec2}

Given the structure of the function $L$ in \eqref{L}, it is crucial to determine the exponent of a desingularizing function $\varphi$ associated with $L$ in Definition~\ref{def:kl}, provided that the exponents of the desingularizing functions associated with $f$, $g$, and $Q$ are already known. The {\em calculus} of PLK exponents and its applications to the linear convergence of first-order methods has been a subject of study in recent years; see, e.g., \cite{li2018calculus}, where explicit convergence rates for various first-order methods are derived and applied to a wide range of optimization models.

Observe to this end that for specific choices of $Q$ and $g$, the structure of $L$ recovers an important class of bifunctions used in the convergence analysis of several first-order methods whose iterates involve {\em momentum terms}. For further details, we refer the reader to, e.g., \cite{boct2016inertial,ochs2014ipiano} with the convergence analysis of certain inertial proximal algorithms and to \cite{chambolle2015convergence} for investigating the convergence of the proximal gradient algorithm with extrapolation. In particular, for the case where $f$ has a PLK exponent $\alpha \in [1/2, 1)$, $g \equiv 0$, and $Q(x, y) = \|x - y\|^2$, it is shown in \cite[Theorem~3.6 and Theorem~5.1]{li2018calculus}, respectively, how to determine the PLK exponent of $L$ and establish the convergence rate of a certain inertial proximal algorithm with constant stepsizes.\vspace*{0.05in}

In this section, we continue the study of PLK exponent calculus associated with the exponent PLK condition from Definition~\ref{def:kl}. The first question to resolve is about the {\em consistency} of the assumptions on $L$ with the {\em existence} of PLK lower exponents. Then, besides considering calculus rules for computing PLK components applied to $L$ \eqref{L}, we raise a new question about the {\em smallest exponent} ensuring the fulfillment of exponent PLK condition and construct some striking examples illustrating this issue.\vspace*{0.05in} 

Let us begin with the consistency question. It is shown in \cite[Theorem~7]{bento2025convergence} that when $\bar x$ is a local  minimizer of the difference function $f_1-f_1$, where $f_1$ is ${\cal C}^1$-smooth around $\bar x$ and where $f_2$ is convex, the fulfillment of the lower exponent PLK  condition for $f_1-f_2$ at $\bar x$ is {\em inconsistent} with the Lipschitz continuity of the gradient mapping $\nabla f_1$ around $\bar x$. Now we construct a one-dimensional example showing that this is {\em not the case} for minimizing structured functions $L$ of type \eqref{L}.

\begin{Example}\label{exa1} {\rm Consider the function $L$ in \eqref{L} with $Q,f: \mathbb{R} \longrightarrow \mathbb{R}$ given by $Q(x): = x^2$, $f(x):=|x|^{3/2}$, and $g \equiv 0$. Observe that $\bar{x}=0$ is a global minimizer of $L$, that $f$ has the  exponent PLK property $q=1/3$, and that $Q$ has the Lipschitz continuous derivative and satisfies the exponent PLK inequality with $q=1/2$. We claim that $L$ satisfies the {\em lower exponent PLK condition} with $1/3$, i.e.,
\begin{equation}\label{even}
\Big|\frac{3}{2}{\rm sign}(x)\cdot|x|^{1/2} + 2x\Big|\ge\frac{1}{5}\Big[|x|^{3/2} + x^2\Big]^{1/3}. 
\end{equation}

\begin{center}
\begin{tikzpicture}
\begin{axis}[
axis lines=middle,
xlabel={$x$},
ylabel={$y$},
legend style={at={(1.05,1)}, anchor=north west},
domain=-5:5,
samples=300,
ymin=0, ymax=3
]
\addplot[blue, thick] {abs(2*x + (3/2)*sign(x)*abs(x)^(1/2))};
\addlegendentry{$|2x + \frac{3}{2} \, \text{sign}(x) \cdot |x|^{1/2}|$}
\addplot[red, thick, dashed] {(1/5)*(x^2 + abs(x)^(3/2))^(1/3)};
\addlegendentry{$\frac{1}{5}(x^2 + |x|^{3/2})^{1/3}$}
\end{axis}
\end{tikzpicture}
\end{center}

Indeed, since the functions on the both sides of \eqref{even} are even, it is sufficient to consider the case where $x>0$. Let
$$
f_1(x): = \left( \frac{3}{2}  x^{1/2} + 2x \right)^3\;\text{ and }\;g_1(x): = x^{3/2} + x^2.
$$
It follows from Newton's binomial theorem that
\[
f_1(x) = \frac{27}{8} x^{3/2} + \frac{27}{2} x^2 + 9 x^{5/2} + 8x^3.
\]
Therefore,  we have in a neighborhood of the origin that
\[
\left( \left| \frac{3}{2} \, \text{sign}(x) |x|^{1/2} + 2x \right| \right)^3 \geq |x|^{3/2} + x^2,
\]
which justifies the claim of this example.}
\end{Example}

To proceed further, recall the results from \cite[Theorem~3.3 and Theorem~3.6]{li2018calculus} providing calculus rules for PLK exponents of block-separable sums of functions satisfying the exponent PLK condition and for the potential function used in the convergence analysis of the inertial
proximal algorithm in \cite{ochs2014ipiano}. This is needed below for our subsequent discussions on the smallest PLK exponents.

\begin{Theorem}\label{funcaocoordenada}
Let $n_j, n \in \mathbb{N}$ for $j = 1, \dots, m$ be such that $\sum_{j=1}^m n_j = n$, and let
\begin{equation}\label{sum}
f(x) = \sum_{j=1}^m f_j(x_j),
\end{equation}
be given in the block separable sum form, where each $f_j$, $j=1,\ldots,m$, is a proper lower semicontinuous function on $\mathbb{R}^{n_j}$ with $x = (x_1, \dots, x_m) \in \mathbb{R}^n$. Suppose further that each $f_j$ is continuous on $\operatorname{dom} \partial f_j$ for $j = 1, \dots, m$ and satisfies the exponent PKL condition with exponent $\alpha_j \in (0,1)$. 
Then the function $f$ satisfies the exponent PLK condition with the exponent calculated by
\[
\alpha = \max\big\{ \alpha_j\;\big|\;j=1,.\ldots,m\big\}.
\]
\end{Theorem}

\begin{Theorem}\label{meiobloco}
Let \( f \) is a proper lower semicontinuous function satisfying the exponent PLK condition at \( \bar{x} \in \operatorname{dom} \, \partial f \) with exponent \( \alpha \in \left[ \tfrac{1}{2}, 1 \right) \), and let \( \beta > 0 \). Consider the proximally perturbed function
\begin{equation}\label{prox}
F(x, y) := f(x) + \frac{\beta}{2} \|x - y\|^2.
\end{equation}
Then \( F \)  satisfies the exponent PLK condition at \( (\bar{x}, \bar{x}) \) with exponent \( \alpha \).
\end{Theorem}

Note that Theorem~\ref{funcaocoordenada} and Theorem~\ref{meiobloco} don't address the question about the {\em minimality} of PLK exponents under combinations. We now discuss this issue in following examples showing, in particular, that the answer depends on which point is chosen. 

\begin{Example}\label{exa2} {\rm In the setting of Theorem~\ref{funcaocoordenada}, select a point $(\bar{x}_1,\ldots,\bar{x}_m) \in \mathbb{R}^n$ such that $\nabla f_j(\bar{x}_j)=0$, i.e., $\bar{x}_j$ is a {\em stationary/critical point} of $f_j$ for each $j=1,\ldots,m$. Let $\alpha_j\in[0,1)$ be the smallest PLK exponent of $f_j$ at $\bar{x}_j$ as $j=1,\ldots, m$. Fix an index $j\in\{1,\ldots,m\}$ and consider the point $(\bar{x}_1, \bar{x}_2,\ldots , x_j, \ldots, \bar{x}_m)\in\mathbb{R}^n$ different from the original one at the $j$-th position. In this case, for all $j\in \{1,\ldots, m\}$, we have
\begin{eqnarray*}
 &\| \nabla f(\bar{x}_1,\ldots , x_j,\ldots, \bar{x}_m)\|= \|\nabla f_j(x_j)\|\ge \big(f_j(x_j) - f_j(\bar{x}_j)\big)^{\alpha_j} \\
 &=
 \big( f(\bar{x}_1,\ldots, x_j,\ldots,\bar{x}_m) -  f(\bar{x}_1,\ldots , \bar{x}_j,\ldots , \bar{x}_m))^{\alpha_j},\quad j=1,\ldots,m.
\end{eqnarray*}
Hence the smallest PLK exponent of $f$ at $\bar{x}= (\bar{x}_1,\ldots, \bar{x}_m)$ is $\max\{\alpha_1,\ldots, \alpha_m\}$. This tells us that the calculation of Theorem~\ref{funcaocoordenada} provides in fact the {\em minimal PLK exponent} for the block separable sum function \eqref{sum} at the critical point.

To illustrate the above, consider the function
$$
f(x,y) = g(x)+h(y)\;\mbox{ with }\;g(x)=|x|^{3/2}\;\mbox{ and }\;h(y)=y^2,\quad (x,y)\in\mathbb{R}^2.
$$
It is known that $g$ and $h$ satisfy the PLK condition at $\bar{x}=0$ and $\bar{y}=0$ with exponent $1/3$ and $1/2$, respectively. By Theorem~\ref{funcaocoordenada}, we get that $f$ satisfies the PLK condition at $(\bar{x}, \bar{y})=(0,0)$ with exponent $q_0 = \max \{1/3,1/2\}=1/2$. Considering now the point $(\bar{x},y)$ with $y\in\mathbb{R}$ gives us  
$$
|2y| = \|\nabla f(0,y)\|\geq M(|0|^{3/2}+y^2)^{q}= M y^{2q}, \quad M>0.
$$
This allows us to conclude that the smallest possible PLK exponent of $f$ in neighborhood of $(0,0)$ is $q=1/2$.}
\end{Example}

The next example demonstrates that the result of Theorem~\ref{funcaocoordenada} {\em may not} give us the {\em minimal PLK exponent} of $f$ in \eqref{sum}.

\begin{Example}\label{exa3} {\rm Let $f\colon\mathbb{R}^2\to\mathbb{R}$ be given in the block separable sum form $f(x, y) = x^2 + y$. We have $\nabla f(x, y) = (2x, 1)$ and hence deduce from \cite[Lemma~2.1]{li2018calculus} that $f$ satisfies the exponent PLK property for every $q \in [0, 1)$ for all point $(\bar{x},\bar{y})\in \mathbb{R}^2$.}
\end{Example} 

Now we show that the counterpart of Theorem~\ref{funcaocoordenada} {\em fails} if the {\em block separable sum structure} of function \eqref{sum} is violated. 

\begin{Example}\label{exe4} {\rm Consider the bifunction $L\colon\mathbb{R}^2\to\mathbb{R}$ defined by
\begin{equation}\label{L1}
L(x,y): = |x|^{3/2} +|y|^{3/2} + |x-y|^2,\quad (x,y)\in\mathbb{R}^2,
\end{equation}
which is not in the block separable sum form \eqref{sum} while looks rather similar. If the counterpart of Theorem~\ref{funcaocoordenada} was satisfied for \eqref{L1}, the corresponding $q$ would be $1/2$. Let us show that this is not the minimal PLK exponent of $L$ in \eqref{L1}. We see that
$$
\nabla L(x,y) = 
\left( \frac{3}{2} \operatorname{sgn}(x)|x|^{1/2} + 2(x - y),\ \frac{3}{2} \operatorname{sgn}(y)|y|^{1/2} - 2(x - y) \right),
$$ 
and, for any $(x,y)$ in a neighborhood of the critical point $(\bar{x}, \bar{y})=(0,0)$, it follows that
$$
\left\| \left( \frac{3}{2} \operatorname{sgn}(x)|x|^{1/2} + 2(x - y),\ \frac{3}{2} \operatorname{sgn}(y)|y|^{1/2} - 2(x - y) \right) \right\|^3 \geq |x|^{3/2} +|y|^{3/2} + |x-y|^2.
$$
Furthermore, we directly calculate that
$$
\left\|  \nabla L(x,y)   \right\|^3
=\left[ \frac{9}{4}(|x| + |y|) + 8(x - y)^2 + 6(x - y)\big({\rm sign}(x)|x|^{1/2} - {\rm sign}(y)|y|^{1/2} \big)\right]^{3/2},
$$
which implies that the exponent PLK condition of $L$ in \eqref{L1} holds with $q=1/3$.
\begin{figure}[H]
\centering
\includegraphics[width=0.6\textwidth]{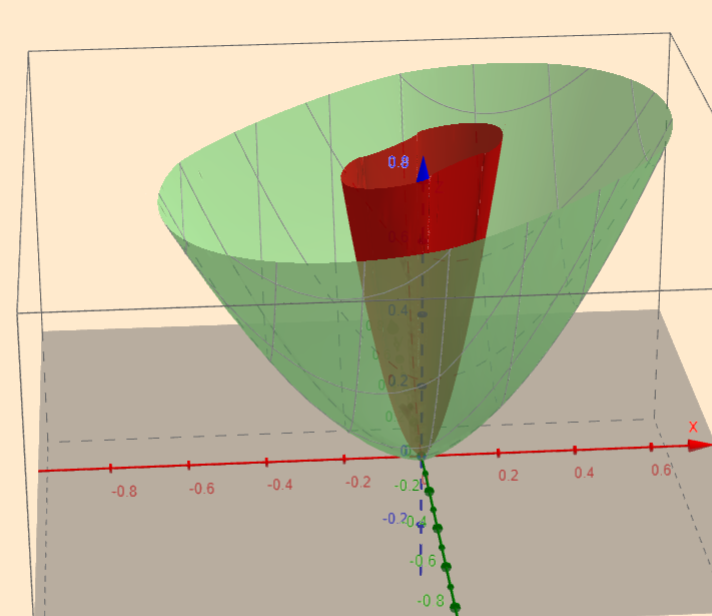} 
\caption{
\protect\tikz[baseline=-0.5ex] \protect\draw[red!70!black, thick] (0,0) -- (0.7,0); $\|\nabla L(x,y)\|^3$ and
\protect\tikz[baseline=-0.5ex] \protect\draw[green!70, thick] (0,0) -- (0.7,0); $L(x,y)$.
}
\end{figure}}
\end{Example}

The last example demonstrates that the {\em proximally perturbed structure} of the bifunction $L$ in Theorem~\ref{meiobloco} is {\em essential} for the theorem conclusion.

\begin{Example}\label{exa5} {\rm Consider the bifunction $L\colon\mathbb{R}^2\to\mathbb{R}$ defined by
\begin{equation*}
L(x,y)=|x| +|y| + |x-y|^2,\quad (x,y)\in\mathbb{R}^2,
\end{equation*}
which resembles \eqref{prox} while not being exactly in that form.
The origin $(\bar{x}, \bar{y})=(0,0)$ is a critical point of $L$.  Picking any $x\neq 0$ and $y\neq 0$, we get 
$$
\nabla L(x,y) = \big({\rm sign}(x) + 2(x-y), {\rm sign}(y) -2(x-y) \big),
$$
which brings us to the relationships
\begin{equation*}
\|\nabla L(x,y)\|^2= 2 +8(x-y)^2 +4(x-y)\big({\rm sign}(x) - {\rm sign}(y)\big)\geq c > 0.
\end{equation*}
\begin{figure}[H]
\centering
\includegraphics[width=0.6\textwidth]{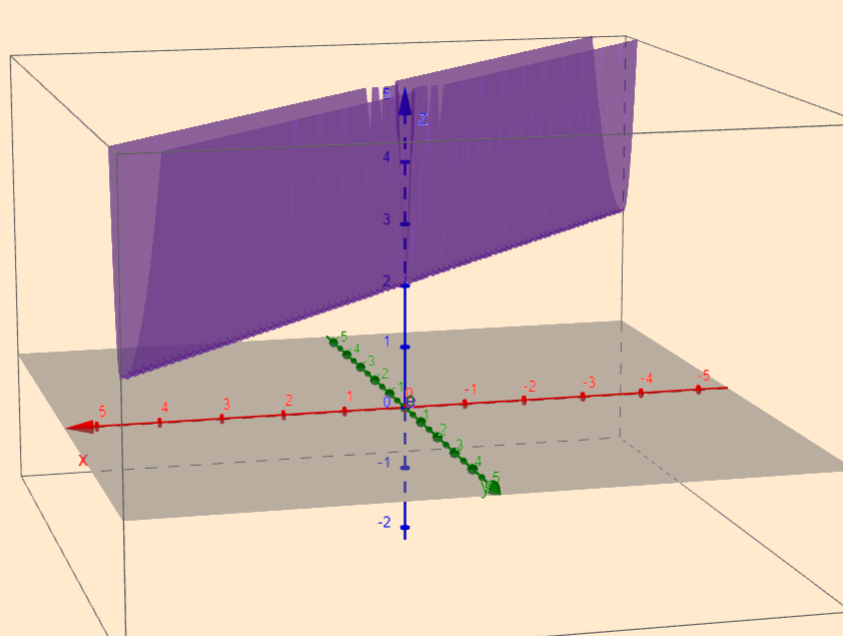}  
\caption{
\protect\tikz[baseline=-0.5ex] \protect\draw[purple!70!black, thick] (0,0) -- (0.5,0); $\| \nabla L(x,y)\|^2$ 
}
\end{figure}
This tells us therefore that for any $q\in[0,1)$ there exist a neighborhood $U$ of the origin and a constant $M>0$ such that
$$
\|\nabla L(x,y)\|\geq c \geq  M\big(|x| +|y| + (x-y)^2\big)^q\;\mbox{ whenever }\;(x,y)\in U,
$$
which indicates the violation of the exponent PLK condition.}
\end{Example}
 
\section{Main Results}\label{sec3}

Recall that the standing assumptions in ${\cal H}$ and ${\cal H}_1$, together with the (basic) PLK condition from Definition~\ref{def:kl}, guarantee the {\em global convergence} of the iterates in the alternating minimization algorithm in \eqref{proximalx}, \eqref{proximaly} to an {\em $M$-stationary/critical point} $(\bar x,\bar y)$ of the cost function $L$ from \eqref{L} satisfying the inclusion $0\in\partial_M L(\bar x,\bar y)$; see \cite[Theorem~9]{attouch2010proximal} and cf.\ \cite[Theorem~1]{bento2025convergence}. In this section, we establish new results about the {\em convergence rates} for the sequences of iterates and cost function values.\vspace*{0.05in}

Before deriving the main results, we present the following useful lemma of its own interest whose device uses the idea that appears within the proof of \cite[Theorem~3.2]{Absil2005}.

\begin{Lemma}\label{propositionAbsil2005}
Let $\{a_k\}\subset[0,\infty)$ be a monotonically decreasing sequenced. Then for all $q\in(0,1)$, we have the estimate
\begin{equation}
\sum_{j=k}^{k+l} \frac{a_j - a_{j+1}}{a_j^q} \leq \frac{1}{1-q}\left(a_k^{1-q} - a_{k+l+1}^{1-q}\right), \quad k, l \in \mathbb{N}.
\end{equation}
\end{Lemma}
\begin{proof}
Since ${a_k} \subset [0,\infty)$ is monotonically decreasing, it follows that $a_j \geq t$ for all $t \in [a_{j+1}, a_j]$, $j\in\mathbb{N}$. This readily yields
\begin{equation*}
\frac{1}{(a_j)^q} \leq \frac{1}{t^q}, \quad t \in [a_{j+1}, a_j], \quad q \in (0,1),
\end{equation*}
which implies therefore that
\begin{equation*}
\frac{a_j - a_{j+1}}{a_j^q} =\int_{a_{j+1}}^{a_j} \frac{1}{a_j^q}dt\leq \int_{a_{j+1}}^{a_j} \frac{1}{t^q}dt.
\end{equation*}
The latter tells us in turn that 
\begin{equation*}
\frac{a_j - a_{j+1}}{a_j^q}\leq\frac{1}{1-q } \left(a_j^{1- q} - a_{j+1}^{1-q}\right), \quad j\in\mathbb{N},
\end{equation*} 
and thus we are done with the proof of the lemma.
\end{proof}

Our first convergence theorem establishes convergence rates for the {\em value sequence} of the alternating minimization algorithm in \eqref{proximalx}, \eqref{proximaly}. 

\begin{Theorem}\label{rateimage}
Suppose that the cost function \( L \) in \eqref{L} satisfies the assumptions in $\mathcal{H}$ and $\mathcal{H}_1$. Let the sequence $\{(x_k, y_k)\}$ converge to \( (x^\ast, y^\ast) \), and let \( L \) satisfy the exponent PLK condition at \( (x^\ast, y^\ast) \) with  
\[
{\varphi}(t) = M t^{1 - q}, \quad q \in [0,1), \quad M > 0.
\]  
Then we have the following convergence rates for the sequence $\{L(x_k,y_k)\}$: 
   
{\bf(i)} If \( q = 0 \), then $\{L(x_k, y_k)\}$ has the finite termination. 

{\bf(ii)} If \( q \in \left(0, \frac{1}{2} \right) \), then $\{L(x_k, y_k)\}$ either has the finite termination or converges superlinearly to $L(x^\ast, y^\ast)$. 

{\bf(iii)} If $q = 1/2$, then $\{L(x_k, y_k)\}$ converges linearly to $L(x^\ast, y^\ast)$.

{\bf(iv)} If \( q \in \left( \frac{1}{2}, 1 \right) \), then there exists \( \gamma > 0 \) such that  
\[
L(x_k, y_k) - L(x^\ast, y^\ast) \leq \gamma k^{-\frac{1}{2q-1}}\;\mbox{ for all }\;k\in\mathbb{N}.
\]
\end{Theorem}
\noindent
{\bf Proof}.
As we know, the sequence $\{(x^k,y^k)\}$ generated by \eqref{proximalx} and \eqref{proximaly} converges to $(\bar{x}, \Bar{y})$, which is an $M$-stationary point of problem \eqref{L}
From \eqref{decrescimo}, we have that the sequence $\{L(x^k,y^k)\}$ is monotonically decreasing. Since $(\bar{x},\bar{y})$ is an accumulation point of $\{(x^k,y^k)\}$, it follows that $\{L(x^k,y^k)\}$ converges to $L(x^\ast,y^\ast)$ as $k\to\infty$.
Defining 
\begin{equation}\label{Psi}
\Psi(x,y):= L(x,y) - L(x^\ast,y^\ast),
\end{equation}
we get that the sequence  of nonnegative numbers 
$\{\Psi(x^k,y^k)\}$ is monotonically decreasing. Combining \eqref{decrescimo} and \eqref{Psi} with
\begin{equation}\label{eq.24}
\|(x_k^* , y_k^*)\| \leq\Big(C + \frac{1}{r_{-}}\Big)\|(x_k, y_k) -  (x_{k-1}, y_{k-1}) \| ,\quad C>0,  
\end{equation}
leads us to the inequalities
\begin{eqnarray}
&\Psi(x_{k-1}, y_{k-1}) - \Psi(x_k, y_k)\geq 
\displaystyle\frac{1}{2 \lambda_{k-1}} \|x_k - x_{k-1}\|^2 + \frac{1}{2 \mu_{k-1}} \|y_k - y_{k-1}\|^2 \, \nonumber \\
&\ge\displaystyle\frac{1}{2 r_+} \|z_k - z_{k-1}\|^2   
\ge\frac{1}{2 r_+(C+ 1/r_- )^2}\|(x^*_k, y^*_k)\|^2\nonumber
\end{eqnarray}
with $(x^*_k, y^*_k) \in \partial L(x_k,y_k)$ and $r_-,r_+$ taken from ${\cal H}_1$. The last inequality above and the imposed exponent PLK condition produce the estimate
\begin{equation}\label{desiq.1}
\Psi(x_{k-1}, y_{k-1}) - \Psi(x_k, y_k)\geq C_0\big( \Psi(x_k, y_k) \big)^{2q},
\end{equation}
where the constant in \eqref{desiq.1} is defined by
\begin{equation*}
C_0:= \frac{1}{2 r_+(C+ 1/r_- )^2 c^2(1-q)^2}>0.
\end{equation*}

Now we are ready to proceed with verifying each assertions in (i)--(iv).\vspace*{0.05in}

{\bf(i)} Let $q=0$. Supposing that the sequence $\{L(x_k,y_k)\}$ is not finitely generated, we get from \eqref{desiq.1} that
$$
\Psi(x_{k-1}, y_{k-1}) - \Psi(x_k, y_k)\geq C_0\;\mbox{ for all large }\;k,
$$
which is a clear contradiction.

{\bf(ii)} Let $q\in (0,1/2)$. Suppose that the sequence 
$\{\L(x_k),y_k)\}$ is not finitely terminated. Then it follows from \eqref{desiq.1} that
\begin{equation}\label{C0}
1+ C_0\big( \Psi(x_k, y_k) \big)^{2q -1 }\leq \frac{\Psi(x_{k-1}, y_{k-1})}{\Psi(x_k, y_k)},
\end{equation}
Since $q\in (0,1/2)$, we have $2q-1<0$ and then deduce from \eqref{C0} and the convergence $\Psi(x_k,y_k)\to 0$ as $k\to\infty$ that
$$
\displaystyle \lim_{k \to \infty}\frac{\Psi(x_{k-1}, y_{k-1})}{\Psi(x_k, y_k)}= \infty.
$$
The latter can be equivalently rewritten as
$$
\lim_{k \to \infty}\frac{\Psi(x_k, y_k)}{\Psi(x_{k-1}, y_{k-1})}= 0,
$$
which yields the claimed superlinear convergence of 
$\{L(x_k,y_k)\}$ to $L(x^*,y^*)$.

{\bf(iii)} For $q = 1/2$, we get from \eqref{desiq.1} that 
$$
\Psi(x_k, y_k)\leq \frac{1}{1+C_0} \Psi(x_{k-1}, y_{k-1}).
$$
Since $1/(1+C_0) \in (0,1)$, assertion (iii) is a consequence of the last inequality.

{\bf(iv)} If $q \in (1/2, 1)$, the estimate in \eqref{desiq.1} tells us that
\begin{equation}\label{C0a}
C_0\leq \frac{\Psi(x_{k-1}, y_{k-1}) - \Psi(x_k, y_k)}{\big( \Psi(x_k, y_k) \big)^{2q}}.
\end{equation}
Denote $\displaystyle \Phi(t):= \frac{1}{t^{2q}}$ and suppose first that there exists $\mu\in (1,\infty)$ such that 
$$
\Phi\big(\Psi(x_k, y_k)\big)\le\mu\,\Phi(\Psi(x_{k-1}, y_{k-1})).
$$
Combining the latter with \eqref{C0a} and Lemma~\ref{propositionAbsil2005} leads us to the relationships
\begin{eqnarray}
\begin{array}{ll}
&C_0 \le \displaystyle\frac{\Psi(x_{k-1}, y_{k-1}) - \Psi(x_k, y_k)}{\big( \Psi(x_k, y_k) \big)^{2q}}=\displaystyle\int_{\Psi(x_k, y_k)}^{\Psi(x_{k-1}, y_{k-1})}\Phi\big(\Psi(x_k, y_k)\big)dt \, \nonumber \\  
&\displaystyle\le\mu\int_{\Psi(x_k, y_k)}^{\Psi(x_{k-1}, y_{k-1})}\Phi\big(\Psi(x_{k-1}, y_{k-1})\big)dt\\
&
\le\displaystyle\frac{\mu}{1-2q} \big[\big(\Psi(x_{k-1}, y_{k-1})\big)^{1-2q}- \big(\Psi(x_k, y_k)\big)^{1-2q}\big] \, \nonumber \\  
& =\displaystyle\frac{\mu}{2q-1}\big[\big(\Psi(x_k, y_k)\big)^{1-2q} -\big(\Psi(x_{k-1}, y_{k-1})\big)^{1-2q}\big],\nonumber
\end{array}
\end{eqnarray}
which imply by $2q-1 > 0$ that
\begin{equation}\label{3.3}
\frac{C_0(2q-1)}{\mu} \leq \big(\Psi(x_k, y_k)\big)^{1-2q} -\big(\Psi(x_{k-1}, y_{k-1})\big)^{1-2q}, \quad k\geq k_0, 
\end{equation}
with some $k_0\in\mathbb{N}$. Letting $j>k_0$, we deduce from \eqref{3.3} that
$$
\sum_{k=k_0}^{j}\frac{C_0(2q-1)}{P} \leq 
\sum_{k=k_0}^{j} \big(\Psi(x_k, y_k)\big)^{1-2q} -\big(\Psi(x_{k-1}, y_{k-1})\big)^{1-2q},
$$
which ensures in turn that
$$
\frac{(j+1-k_0)C_0(2q-1)}{\mu} \leq \big(\Psi(x_j,y_j)\big)^{1-2q} - \big(\Psi(x_{k_0},y_{k_0})\big)^{1-2q}.
$$
Consequently, we get the estimate
$$
\big(\Psi(x_j,y_j)\big)^{1-2q} \geq \frac{(j+1-k_0)C_0(2q-1)}{\mu} + \big(\Psi(x_{k_0},y_{k_0})\big)^{1-2q}.
$$
Since the function $ t\mapsto t^{-\frac{1}{2q-1}}$ is decreasing  for $q\in (1/2,1)$, it follows that
$$
\Psi(x_j,y_j) \leq \left(\frac{(j+1-k_0)C_0(2q-1)}{\mu} + \big(\Psi(x_{k_0},y_{k_0})\big)^{1-2q}\right)^{-\frac{1}{2q-1}}.
$$
This allows us to find a number $\eta>0$ such that
$$
\Psi(x_j,y_j) \leq \eta j^{-\frac{1}{2q-1}},  
$$
which therefore confirms the claimed convergence rate in (iv) for this case.

Now we examine the remaining situation in {\bf(iv)} where the number
$\mu$ mentioned above doesn't exist, i.e., for any
$\mu\in(1,\infty)$ we have 
$$
\Phi\big(\Psi(x_{k},y_{k})\big)>\mu\,\Phi\big(\Psi(x_{k-1}, y_{k-1})\big)\;\mbox{ when }\;k\;\mbox{ is sufficiently large}.
$$
Fix $\mu\in(1,\infty)$ and define $\mu_1: = \frac{1}{P^{\frac{1}{2q}}}$. Then it follows from the definitions that $\Psi(x_{k},y_{k})\le\mu_1 \Psi(x_{k-1},y_{k-1})$ and therefore 
$$
\Psi(x_{k},y_{k})^{1-2q}\ge\mu_1^{1-2q}\Psi(x_{k-1},y_{k-1})^{1-2q}
$$ 
since $1-2q<0$. Subtracting $\Psi(x_{k-1},y_{k-1})^{1-2q}$
on both sides of the last inequality gives us the estimate
$$
\Psi(x_{k},y_{k})^{1-2q} - \Psi(x_{k-1},y_{k-1})^{1-2q} \geq \big(\mu_1^{1-2q} -1\big)\Psi(x_{k-1},y_{k-1})^{1-2q}. 
$$
Note that $\mu_1 \in (0,1)$ and therefore $\mu_1^{1-2q}>1$ for all $q\in(1/2,1)$. The convergence $\Psi(x_{k},y_{k})\to 0$ yields  $\psi(x^k)^{1-2q}\to\infty$ as $k\to\infty$, and hence
$$
\Psi(x_{k},y_{k})^{1-2q} - \Psi(x_{k-1},y_{k-1})^{1-2q} \geq \frac{C_0(2q-1)}{\mu}
$$
for large $k$. The rest of the proof is similar to the above arguments based on  \eqref{3.3}, and thus we are done with verifying (iv) and the entire theorem. $\h$ \vspace*{0.08in}

The second theorem in this section provides convergence rates for the {\em sequence of iterates} $\{x_k\}$ of the alternating minimization algorithm in \eqref{proximalx} and \eqref{proximaly} under the exponent PKL condition. The main improvement is in the lower exponent assertion (ii), and we'll concentrate on its proof.

\begin{Theorem}\label{iter}  
Suppose that the assumptions in $\mathcal{H}$ and $\mathcal{H}_1$ are fulfilled for the cost function \( L \)  in \eqref{L}, that the sequence of iterates $\{(x_k,y_k)\}$ in the alternating minimization algorithm \eqref{proximalx}, \eqref{proximaly} converges to the $M$-stationary point \( (x^\ast, y^\ast) \)  of \eqref{L}, and that \( L \) satisfies the exponent PLK property at \( (x^\ast, y^\ast) \) with  
\[
\varphi(t) = M t^{1 - q}, \quad q \in [0,1), \quad M > 0.
\]  
Then the following assertions hold:  
   
{\bf(i)} If \( q = 0 \), then the sequence $\{(x_k, y_k)\}$ converges in a finite number of steps. 

{\bf(ii)} If \( q \in \left(0, \frac{1}{2} \right) \), then $(x_k, y_k)$ either has the finite termination or converges superlinearly to $(x^\ast, y^\ast)$. 

{\bf(iii)} If $q = 1/2$, then $\{(x_k, y_\})$ converges linearly to $(x^\ast, y^\ast)$.

{\bf(iv)} If \( q \in \left( \frac{1}{2}, 1 \right) \), then there exists \( \gamma > 0 \) such that  
\[
(x_k, y_k) - (x^\ast, y^\ast) \leq \gamma k^{-\frac{1-q}{2q-1}}.
\]
\end{Theorem}
{\bf Proof}. Assertions (i), (iii), and (iv) follow from \cite[Theorem~11]{attouch2010proximal}, where a linear convergence version of (ii) is also obtained. Now we do much better as formulated in (ii). It follows from \cite[Theorem~3.1]{attouch2010proximal} that
\begin{equation}\label{3.33}
\varphi\big(L(x_k, y_k) - L(x^\ast, y^\ast)\big) - \varphi\big(L(x_{k+1}, y_{k+1}) - L(x^\ast, y^\ast)\big) \geq \frac{1}{\vartheta} \frac{\|z_{k+1} - z_k\|^2}{\|z_k - z_{k-1}\|},
\end{equation}
where $\vartheta:=2r_+(C + 1/r_-)$ and $z_k = (x_k, y_k)$ with some $C>0$. Given $r\in (0, 1)$, we have the following two possibilities for each $k\in\mathbb{N}$:\\[0.5ex]
\noindent
{\bf(a)} $\|z_{k+1} - z_k\| \geq r\|z_k - z_{k-1}\|$.\\[0.5ex]
\noindent
{\bf(b)} $\|z_{k+1} - z_k\| < r\|z_k - z_{k-1}\|$.\vspace*{0.08in}

In case {\bf(a)}, it follows from \eqref{3.33} with $\varphi(t)= Mt^{1-q}$ with the notation $\Psi(x,y):= L(x,y) - L(x^\ast,y^\ast)$ that
\begin{equation*}
\begin{array}{ll}
\|z_{k+1} - z_k\|&\le\displaystyle\frac{\vartheta M}{r} \Psi(z_k)^{1-q} - \displaystyle\frac{\vartheta M}{r}\Psi(z_{k+1})^{1-q}\\
&\leq  r\|z_k - z_{k-1}\| +\displaystyle\frac{\vartheta M}{r}\Psi(z_{k})^{1-q} -\displaystyle\frac{\vartheta M}{r}\Psi(z_{k+1})^{1-q}.
\end{array}
\end{equation*}
Conversely, in case {\bf(b)}, we immediately get
$$
\|z_{k+1} - z_k\| \leq  r\|z_k - z_{k-1}\| \leq r\|z_k - z_{k-1}\| + \frac{\vartheta M}{r}\Psi(z_{k})^{1-q} - \frac{\vartheta M}{r}\Psi(z_{k+1})^{1-q}.
$$
Therefore, in both cases it holds that
\begin{equation}
\|z_{k+1} - z_k\| \leq r\|z_k - z_{k-1}\| + \frac{\vartheta M}{r}\Psi(z_{k})^{1-q} - \frac{\vartheta M}{r}\Psi(z_{k+1})^{1-q}.
\end{equation}
The latter readily implies that for any $\l\in\mathbb{N}$ we have
\begin{equation}\label{somatorio3}
\sum_{j=k}^{k+l}\| z_{k+1} - z_k\| \leq \frac{r}{1-r}\|z_k - z_{k-1} \| + \frac{\vartheta M}{r(1-r)}\big(\Psi(z_{k})^{1-q}  - \Psi(z_{k+l+1})^{1-q}\big). 
\end{equation}
Define further the value (which is a finite number under the basic PLK condition as follows from Lemma~\ref{basic-lem}(ii))
$$
s_k:= \sum_{j=k}^{\infty}\|z_{k+1} - z_k\|.
$$
Passing to the limit  $l\to\infty$ in \eqref{somatorio3} leads us to
$$
s_k  \leq \frac{r}{(1-r)}\|z_k - z_{k-1} \| + \frac{\vartheta M}{r(1-r)} \Psi(z_{k})^{1-q}.
$$
Using the obtained inequality together with \eqref{decrescimo} and  $L(z_{k-1})- L(z_k)\le L(z_{k-1})-L(x^\ast,y^\ast)=\Psi(z_{k-1})$ gives us the estimate
\begin{equation}\label{eq:1111PA}
s_k \leq \frac{r}{(1-r)} \big(\Psi(z_{k-1})\big)^{1/2} + \frac{\vartheta M}{r(1-r)} \Psi(z_{k})^{1-q}. 
\end{equation}
When $q\in(0,1/2)$, it follows from \eqref{eq:1111PA} that 
\begin{equation}
s_k \leq \frac{r}{(1-r)} \big(\Psi(z_{k-1})\big)^{1/2} + \frac{\vartheta M}{r(1-r)} \Psi(z_{k})^{1/2}. 
\end{equation}
Since the sequence $\{\psi(x^k)\}$ is monotonically decreasing, we get 
$$
s_k \leq \frac{r^2 + \vartheta M}{(1-r)} \big(\Psi(z_{k-1})\big)^{1/2},
$$
and the result follows from Theorem~\ref{rateimage}(ii), which completes the proof.
$\h$\vspace*{0.05in}

\section{Discussions on Applications to Noncooperative Games and Behavioral Science}\label{sec:appl}

As mentioned in Section~\ref{sec1}, {\em convergence rate} of numerical algorithms are of crucial importance to solve various classes of {\em noncooperative games} in game theory and applications. From this viewpoint, the convergence results obtained in Section~\ref{sec3} for the alternating minimization algorithm present valuable refinements in practically important models of {\em alternating games} with symmetric or asymmetric costs to move. We refer the reader to, e.g., Cruz Neto et al. \cite{neto2013learning} and Soubeyran et al. \cite{soubeyran2019variational} for some models of {\em worthwhile-to-move potential games} played in asymmetric metric spaces, where each agent plays, in alternation, a worthwhile move given the action done by the other. The central concept of such behavioral games is that of an {\em individual worthwhile move} that
improves a player's payoff without requiring a too high cost to move if other players stay at the status quo. That is, it better satisfies the needs and desires of this individual without too many sacrifices for himself or herself. This concept is borrowed from the {\em variational rationality approach} of
individual and social stay and change human dynamics; see, e.g., \cite{soubeyran2009variational,soubeyran2010variational} and other developments documented in
$$
{\rm https://sites.google.com/view/antoine-soubeyran}
$$
for applications of different optimization algorithms to models behavioral science.

The obtained results on fast convergence and especially on the finite termination for the alternating optimization algorithm are of {\em striking importance} for applications to such models. They show that the finite termination holds when the payoff functions of two players are sharp enough concerning the worthwhile change in their own actions. A lot of applications to game theory can be expected in both static and dynamic frameworks. The concept of the {\em existence of temporary/permanent traps} will be among central topics of our future research; see \cite{ms19} for the unified approach of variational analysis and variational stationarity to these issues.
New potential functions will be used to get novel results on convergence rates under low exponent PLK conditions.

\bmhead{Acknowledgements}
Research of Boris Mordukhovich was supported by the US National Science Foundation under grant DMS-2204519 and by the Australian Research Council under Discovery Project DP250101112. Research of Glaydston Bento was supported by CNPq grants 314106/2020-0. Research of Antoine Soubeyran  was supported by the French National Research Agency Grant ANR-17-EURE-0020 and by the Excellence Initiative of Aix-Marseille University: A*MIDEX.

\section*{Declarations}
The authors declare that they have no conflict of interest.

\end{document}